\documentclass[10pt]{amsart}

\newcounter{counterc}[section]

\newcounter{countert}[section]

\newcounter{counterd}[section]

\newcounter{counterl}[section]

\newtheorem{theorem}[countert]{Theorem}
\newtheorem{lemma}[counterl]{Lemma}
\newtheorem{conjecture}[counterc]{Conjecture}

\newtheorem{remark}{Remark}[section]

\newtheorem{claim}{Claim}[section]
\newtheorem{observation}{Observation}[section]
\theoremstyle{definition}
\newtheorem{definition}[counterd]{Definition}

\usepackage[dvips]{graphicx}

\newcommand\kplus[1]{$#1^+$}
\newcommand\kminus[1]{$#1^-$}

\newcommand{\sst}[2]{\left\{#1\,\middle|\,#2\right\}}
\def\abs#1{|#1|}

\newcommand{\mad}{mad}
\title{Equitable coloring of sparse planar graphs}

\author{Rong Luo}
\address{Department of Mathematical Sciences, Middle Tennessee State University, Murfreesboro, TN  37132  \texttt{rluo@mtsu.edu}. Research supported in part by MTSU summer and academic year research grant 2008-2009.}
\author{Jean-S\'ebastien Sereni}
\address{C.N.R.S. (LIAFA, Universit\'e Denis Diderot),
Paris, France, and Department of Applied Mathematics (KAM), Faculty of
Mathematics and Physics, Charles University, Prague, Czech Republic. \texttt{sereni@kam.mff.cuni.cz}.}
\author{D. Christopher Stephens}
\address{Department of Mathematical Sciences, Middle Tennessee State University, Murfreesboro, TN  37132  \texttt{cstephen@mtsu.edu}.  Research supported in part by NSA grant H98230-07-1-0006.}
\author{Gexin Yu}
\address{Department of Mathematics, College of William and Mary, Williamsburg, VA,  23185
\texttt{gyu@wm.edu}. Research supported in part by the NSF grant DMS-0852452.}

\begin{document}
\maketitle

\begin{abstract}
A proper vertex coloring of a graph $G$ is equitable if the sizes of
color classes differ by at most one. The equitable chromatic
threshold $\chi_{eq}^*(G)$ of $G$ is the smallest integer $m$ such
that $G$ is equitably $n$-colorable for all $n\ge m$.   We show that for planar graphs $G$ with minimum degree at least two,  $\chi_{eq}^*(G)\le 4$ if the girth of $G$ is at least $10$, and  $\chi_{eq}^*(G)\le 3$ if the girth of $G$ is at least $14$.
\end{abstract}

\pagestyle{myheadings}
\thispagestyle{plain}

\section{Introduction}
Graph coloring is a natural model for scheduling problems. Given a graph $G=(V, E)$,  a
proper vertex $k$-coloring is a mapping $f\colon V(G)\to \{1, 2, \ldots, k\}$ such that $f(u)\not=f(v)$ if $uv\in E(G)$.
The notion of equitable coloring is a
model to equally distribute resources in a scheduling problem.  A proper $k$-coloring $f$ is equitable if
\[
\abs{V_1}\le \abs{V_2}\le \ldots\le \abs{V_k}\le \abs{V_1}+1
\]
where $V_i=f^{-1}\left(i\right)$ for $i\in\{1,2,\ldots,k\}$.

The \emph{equitable chromatic number $\chi_{eq}(G)$} of $G$
is the smallest integer $m$ such that $G$ is equitably
$m$-colorable. The \emph{equitable chromatic threshold} of $G$, denoted by
$\chi_{eq}^*(G)$, is the smallest integer $m$ such that $G$ is
equitably $n$-colorable for all $n\ge m$. Note that
$\chi_{eq}(G)\le \chi_{eq}^*(G)$ for every graph $G$, and the two values may be
different: for example, $\chi_{eq}(K_{7,7})=2$ while
$\chi_{eq}^*(K_{7,7})=8$.

Hajnal and Szemer\'edi \cite{HS70} proved that $\chi_{eq}^*(G)\le
\Delta(G)+1$ for any graph $G$ with  maximum degree $\Delta(G)$.
The following conjecture made by Chen, Lih and Wu \cite{CLW94}, if
true, strengthens the above result.
\begin{conjecture}[Chen, Lih and Wu \cite{CLW94}]
For any connected graph $G$ different from $K_m, C_{2m+1}$ and $K_{2m+1, 2m+1}$, $\chi_{eq}^*(G)\le \Delta(G)$.
\end{conjecture}
Except for some special cases, the conjecture is still wide open in general.

Another direction of research on equitable coloring is to consider special families of graphs.
For planar graphs, Zhang and Yap \cite{ZY98} proved that a planar
graph is equitably $m$-colorable if  $m\ge \Delta(G)\ge 13$. When the girth $g(G)$ is large, fewer colors are needed.
\begin{theorem}[Wu and Wang, \cite{ww08}]\label{ww-thm}
Let $G$ be a planar graph with $\delta(G)\ge 2$. \\
(a) If $g(G)\ge 26$, then $\chi_{eq}^*(G)\le 3$;\\
(b) If $g(G)\ge 14$, then $\chi_{eq}^*(G)\le 4$.
\end{theorem}
The purpose of this paper is to improve the above two results.  Our
main results are contained in the following theorems.

\begin{theorem}\label{thm:4-coloring}
If $G$ is a planar graph with $\delta(G)\ge 2$ and $g(G)\ge 10$, then $\chi_{eq}^*(G)\le 4$.
\end{theorem}
\begin{theorem}\label{thm:3-coloring}
If $G$ is a planar graph with $\delta(G)\ge 2$ and $g(G)\ge 14$, then
$\chi_{eq}^*(G)\le 3$.
\end{theorem}

Since $K_{1, n}$ is not equitably $k$-colorable when $n\ge 2k-1$, we
cannot drop  the requirement of $\delta(G)\ge 2$ in the theorems.
On the other hand, we do not believe that the girth conditions are
best possible.   Note that $\chi_{eq}^*(K_{2,n})=\lceil\frac{n}{3}\rceil+1$ for $n\ge 2$ and the girth of $K_{2,n}$ is $4$.    It would be interesting to find the best possible girth condition for both 3- and 4-equitable colorings.

Last, let us note that actually, we do not use planarity but only the weaker assumption that the graphs have maximum average degree less than $2.5$ for Theorem~\ref{thm:4-coloring}, and less than
$7/3$ for Theorem~\ref{thm:3-coloring}. 
We do need the girth conditions, however,  not only to control the density, but also to ensure the minimum degree is at least $2$ at all times when we do reductions.

\section{Preliminaries}
Before starting, we introduce some notation.
In the whole paper, we take ${1,2,\ldots,m}$ to be the set of
integers modulo $m$.
A \emph{$k$-vertex} is a vertex
of degree $k$; a \kplus{k}- and  a \emph{\kminus{k}-vertex} have degree at least and
at most $k$, respectively.  A {\it thread} is either (a) a path with 2-vertices
in its interior and \kplus{3}-vertices as its endvertices, or
(b) a cycle with exactly one \kplus{3}-vertex, and all other vertices of degree 2 (in other words, case (a) with endvertices equal).
A \emph{$k$-thread} 
has $k$ interior 2-vertices.  If a \kplus{3}-vertex $u$ is the endvertex
of a thread containing a 2-vertex $v$,   and the distance between
$u$ and $v$ on the  thread is $l+1$, then we say that $u$ and $v$
are \emph{loosely $l$-adjacent}. Thus ``loosely $0$-adjacent'' is the
same as the usual ``adjacent.''

All of our proofs rely on the techniques of reducibility and
discharging.     We start with a minimal counterexample $G$ to the
theorem we are proving, and the idea of the reduction is as follows.
We remove a small subgraph $H$  (for instance, a vertex of degree at
least three, together with its incident 2-threads) from the graph
$G$. By the minimality of $G$, we therefore have an equitable $k$-coloring $f$ of $G-H$, and we
attempt to extend $f$ to an equitable coloring of $G$.  This can be
done if we can equitably $k$-color $H$ itself, with some extra
conditions: namely, the color classes which should
be ``large'' in $H$ are predetermined by the existing coloring of
$G-H$; and secondly, the parts of $H$ with edges to $G-H$ have color
restrictions. If every equitable $k$-coloring of $G-H$ can be extended into an equitable $k$-coloring of $G$,
then $H$ is called a \emph{reducible configuration}.

We will handle the latter condition by means of lists of allowed
colors in $H$.  We will handle the former condition by
predetermining the sizes of the color classes.  Thus we have the following definition. 
\begin{definition}
Let $H$ be a graph with list assignment $L=\{l_v\}$, with $l_v\subseteq\{1,2,\dots,m\}$.
The graph $H$ is \emph{descending-equitably $L$-colorable} if $H$ can be $L$-colored such
 that $\abs{V_1}\geq \abs{V_2} \geq \ldots \geq \abs{V_m}\geq \abs{V_1}-1$.
\end{definition}
Note that if $G-H$ has an equitable $k$-coloring with $\abs{V_1}\leq \abs{V_2} \leq \ldots \leq \abs{V_k}\leq \abs{V_1}+1$, then $G$ is equitably $k$-colorable if $H$ is descending-equitably $L$-colorable.
Because of this, a descending-equitably $L$-colorable subgraph $H$ is a reducible configuration in $G$. 
Regarding the lists $\ell_v$, we always take $\ell_v$ to be the set of all
colors not assigned to any neighbor of $v$ in $G-H$.

The \emph{maximum average degree} of $G$ is
$\mad(G)=\max\sst{\frac{2\abs{E(H)}}{\abs{V(H)}}}{H\subseteq G}$.
A planar graph $G$ with girth at least $g$ has maximum average
degree less than $\frac{2g}{g-2}$. We let the initial charge at vertex
$v$ be $M(v)=d(v) - \frac{2g}{g-2}$. We will introduce some rules to
re-distribute the charges (discharging), and after the discharging
process, every vertex $v$ has a final charge $M'(v)$. Note that
\begin{equation}\label{eq-0}
\sum_{v\in V(G)}M'(v)=\sum_{v\in V(G)} M(v)=\sum_{v\in V(G)} \left(d(v)-
\frac{2g}{g-2}\right)<0.
\end{equation}
We will show that either we have some reducible configurations, or
the final  charges are all non-negative.   The former contradicts
the assumption that $G$ is a counterexample, and the latter
contradicts (\ref{eq-0}). 

We will prove the theorems on $3$-coloring and $4$-coloring
separately. Before the proofs,  we provide some properties useful
to equitable $m$-coloring with $m\ge 3$.

Let $m\geq 3$ be an integer. Let $G$ be a graph that is not
equitably $m$-colorable with $\abs{V} + \abs{E}$ as small as possible.
\begin{observation}\label{claim-1.1}
The graph $G$ is connected.
\end{observation}
\begin{proof} Let $H_1, H_2,\ldots, H_k$ be the connected
components of $G$, where $ k \geq 2$.  By the minimality of $G$, both $H
= H_1\cup H_2\cup\ldots \cup H_{k-1}$ and $H_k$ are equitably $m$-colorable.
An equitable $m$-coloring of $H$ with $\abs{V_1(H)} \geq \abs{V_2(H)} \geq
\ldots \geq  \abs{V_m(H)}$ and an equitable $m$-coloring of $H_k$ with
$\abs{V_1(H_k)}\leq \abs{V_2(H_k)} \leq \ldots \leq \abs{V_m(H_k)}$ induce an
equitable $m$-coloring of $G$, which contradicts the choice of $G$.
\end{proof}

\section{Equitable 4-coloring}
\refstepcounter{counterl}
In this section, we prove Theorem~\ref{thm:4-coloring}. We start with some useful lemmas.
\begin{lemma}\label{le:4-5-thread}
Let G be a graph and $P = y_0y_1\ldots y_ty_{t+1}$ such that $t\in \{4, 5\}$,  and  $d(y_i) = 2$ for each $i\in\{1,2,\ldots,t\}$.
Let $m \geq 4$ be an integer and $a,b \in  \{1,2,\ldots,m\}$. Let $x$ be  an arbitrary vertex  in $\{y_1,y_2,\ldots,y_t\}$. If $G - \{y_1,\ldots,y_t\}$ has an equitable $m$-coloring $f$, then $f$ can be extended to an equitable $m$-coloring of $G$ such that
$f(x) \notin\{a,b\}$ unless $m = 4$, $t = 5$, and $x\in\{y_2,y_4\}$.
\end{lemma}

\begin{proof}
Let $V_1, \ldots, V_m$ be the $m$ color classes of $G-S$ under $f$ with $\abs{V_1} \leq
\abs{V_2} \leq \ldots \leq \abs{V_m}$, where $S=\{y_1, y_2, \ldots, y_t\}$.
By symmetry, we may assume that $x = y_i$ with $i \leq \lceil\frac{t}{2}\rceil$.

When $m\ge t$, we extend~$f$ to~$G$ using each color in~$\{1,\dotsc,t\}$ exactly
once. Such an extension is an equitable $m$-coloring of~$G$ provided that
$f(y_1)\neq f(y_0)$ and~$f(y_t)\neq f(y_{t+1})$. We just need to take care that in
addition $f(x)\notin\{a,b\}$. To do se, arrange the vertices $y_j$ into a list
$x_1,\dotsc,x_t$ such that $x_1=x$ and $x_t\notin\{y_1,y_t\}$. We now
color~$x_1,\dotsc,x_t$ greedily (in order) with colors in~$\{1,\dotsc,t\}$, never
using a color twice. For~$x_1$, we have at least $t-3\ge1$ choices.  For
each~$x_j$ with $2\le j\le t-1$, we have at least $t-(t-2)-1=1$ choice. Finally, we have
at least $t-(t-1)=1$ choice also for $x_t$, because we took care that
$x_t\notin\{y_1,y_t\}$.

If $m<t$, then $m=4$ and $t=5$, and in this case, $x\in \{y_1, y_3\}$. If $1 \notin\{a,b,f(y_0)\}$, then assign $1$
to $y_1$ and $y_3$, assign a color $c\in \{2,3,4\}\setminus\{f(y_6)\}$ to $y_5$, and assign the other two colors in
$\{2,3,4\}\setminus\{c\}$ arbitrarily to $y_2$ and $y_4$.  
If $ 1 \in \{f(y_0), a,b\}$, then $\abs{\{2,3,4\} - \{f(y_0), a,
b\}}\geq 1$. Let $x'\in\{y_1,y_3\}\setminus\{x\}$, and $c_2 \in \{2,3,4\}\setminus\{f(y_0), a,b\}$.
Assign $1$ to $y_2$ and $y_4$, assign $c_2$ to $x$, assign a color $c_3 \in \{2,3,4\}\setminus\{c_2,f(y_0)\}$ to $x'$, and assign the
remaining color $c_4 \in \{2,3,4\}\setminus\{c_2,c_3\}$ to $y_5$. If $c_4=f(y_6)$, then swap colors on $y_5$ and $y_4$, i.e. recolor $y_4$ with $c_4$ and $y_5$ with $1$.  In either case, $f$ is extended to $G$, a contradiction.
\end{proof}

\begin{lemma}
\label{le:2-thread}
Let $xy_1y_2y$ be a 2-thread of $G$ and $m \geq 4$ be an integer.
If $G-\{y_1,y_2\}$ has an $m$-equitable coloring $f$ such that $f(x)  \not = f(y)$, then
$f$ can be extended to an equitable $m$-coloring of $G$.
\end{lemma}
\begin{proof}  
Let $f$ be an equitable $m$-coloring of $G -
\{y_1,y_2\}$ and let $V_1, \ldots, V_m$ be the $m$ color classes
with $\abs{V_1} \leq \abs{V_2} \leq \ldots \leq \abs{V_m}$.  If $f(x) \not =
f(y)$, then there is a bijection $\phi : \{1,2\} \to \{1,2\}$
such that $\phi(1) \not = f(x)$ and $\phi(2) \not = f(y)$. Assign
$\phi(1)$ to $y_1$ and $\phi(2)$ to $y_2$.  Hence $f$ can be
extended to $G$.
\end{proof}

\begin{lemma}\label{le:2-1-thread} 
Let $xy_1y_2y$ be a 2-thread  and $xy_3z$ be a
$1$-thread incident with $x$. Let $m\geq 4$ be an integer. If $G -
\{y_1,y_2,y_3\}$ has an equitable $m$-coloring $f$ with $f(x) \not
\in \{f(y), f(z)\}$, then $f$ can be extended to an equitable $m$-coloring of $G$.
\end{lemma}
\begin{proof}
Let $V_1,V_2, \ldots, V_m$ be the $m$ color classes with $\abs{V_1} \leq
\abs{V_2} \leq \ldots \leq \abs{V_m}$. If $f(x) \in \{1,2,3\}$, then let $a =
f(x)$, so $a\neq f(y)$. Otherwise, let $a \in
\{1,2,3\}\setminus\{f(y)\}$. Let $b \in \{1,2,3\}\setminus\{a,f(z)\}$,
and $c \in \{1,2,3\}\setminus\{a,b\}$. Then $b \notin\{f(x), f(z)\}$,
$c \neq f(x)$, and $\{a,b,c\}=\{1,2,3\}$. Assigning $a$ to $y_2$, $b$ to
$y_3$ and $c$ to $y_1$ yields an equitable $m$-coloring of
$G$, a contradiction.
\end{proof}

\noindent\textbf{Proof of  Theorem~\ref{thm:4-coloring}.}
Let $G$ be a minimal counterexample to Theorem~\ref{thm:4-coloring} with $\abs{V} + \abs{E}$
as small as possible. That is, $G$ is a planar graph with $\delta(G)\ge2$
and girth at least $10$,
and $G$ is not equitably $m$-colorable for some integer $m\ge 4$ but every proper subgraph of $G$ with minimum
degree at least $2$ is equitably $m$-colorable for each $m\geq 4$.

\begin{claim}\label{claim-3.2}
The graph $G$ has no $t$-thread with $t \geq 3$, and $G$ has no
thread whose endvertices are identical.
\end{claim}

\begin{proof}
Suppose on the contrary that $G$ has a $t$-tread $P = v_0v_1\ldots v_t v_{t+1}$ with $t \geq 3$, where $d(v_0),d(v_{t+1})\ge 3$.

If $v_0 \not = v_{t+1}$ or $d({v_0})\geq 4$, consider $G_1= G -
\{v_1,\ldots, v_t\}$. Then $\delta(G_1) \geq 2$. By the minimality of
$G$, the graph $G_1$ has an equitable $m$-coloring.  Let $V_1,V_2, \ldots, V_m$ be
the $m$ color classes with $\abs{V_1} \leq \abs{V_2} \leq \ldots \leq
\abs{V_m}$. We can extend the coloring  to $G$ to obtain an equitable
$m$-coloring of $G$ as follows: first color the vertex $v_i$ by the
color $k$ where $ k \equiv i \pmod{m}$ for each $i\in\{1,2,\ldots,t\}$. Swap the colors of $v_1$ and $v_2$ if the colors of
$v_1$ and $v_0$ are the same, and further swap the colors of $v_{t-1}$ and $v_t$ if the colors of $v_t$ and $v_{t+1}$ are the same.

Now assume that $v_0 = v_{t+1}$ and $d(v_0) = 3$. Let $x \in N(v_0)\setminus\{v_1,v_t\}$.
If $d(x) \geq 3$, consider $G_2= G - \{v_0, v_1,\ldots, v_t\}$. Then $\delta(G_2) \geq 2$. By the choice
of $G$, the graph $G_2$ has an equitable $m$-coloring with color classes $V_1,V_2,\ldots,V_m$
such that $\abs{V_1} \leq \abs{V_2} \leq \ldots \leq
\abs{V_m}$.  Since $q\ge 1$, we can extend the coloring  to $G$ to obtain an equitable
$m$-coloring of $G$ as follows: first color the vertex $v_i$ by the
color $k$ where $ k \equiv i \pmod{m}$; if $0 \equiv t \pmod{m}$, swap the
colors of $v_t$ and $v_{t-1}$; if the colors of $x$ and $v_0$ are the
same, further swap the colors of $v_0$ and $v_i$, where $i\in\{1,2\}$ such that the color of $v_i$ is different from that of $v_t$ (such
a vertex $v_i$ exists since $v_0$, $v_1$ and $v_2$ are colored differently).

If $d(x) = 2$, then let $Q = x_0x_1\ldots x_qx_{q+1}$ be the thread
containing the edge $v_0x_1$ where $x_1 = x$ and $x_0 = v_0$.
Consider the graph $G_3 = G - \{v_0, x_1, \ldots, x_q, v_1,\ldots,
v_t\}$. Then $\delta(G_3) \geq 2$. By the minimality of $G$, the graph $G_3$ has
an equitable $m$-coloring with color classes $V_1,V_2, \ldots, V_m$ such that
$\abs{V_1} \leq \abs{V_2} \leq \ldots \leq \abs{V_m}$. We first
extend the coloring  to $G_2$ to obtain an equitable $m$-coloring of
$G - \{v_1,\ldots, v_t\}$ as follows: first color the vertex
$x_i$ by the color $k$ where $k \equiv {i+1} \pmod{m}$ for each
$i\in\{0, 1, \ldots q\}$; if  $x_q$ and $x_{q+1}$ have the same color,
swap the colors of $x_{q}$ and $x_{q-1}$. Next, we further extend the
coloring to $G$ similarly to the case that $d(v_0) \geq 4$.
\end{proof}

Let $x$ be a vertex of degree $d = d(x)\geq 3$. Then $x$
is  the endvertex of $d$ threads. Set $T(x) = (a_2,a_1,a_0)$ where $a_i$ is
the number of $i$-threads incident with $x$. Let $t(x)= 2a_2 + a_1$.
Claim~\ref{claim-3.2} implies that $t(x)$ is the number of $2$-vertices loosely adjacent to $x$. 

\begin{claim}\label{claim-3.3}
If $x$ is a 4-vertex, then $t(x) \leq 5$.
\end{claim}

\begin{proof}
Suppose on the contrary that $x$ is a $4$-vertex with $t(x) \geq 6$. Claim~\ref{claim-3.2} implies that $x$ is
not incident with any $t$-thread such that $t \geq 3$. Since $t(x) \geq
6$, the vertex $x$ is incident with at least two $2$-threads. Label two $2$-threads
incident with $x$ as $xx_1z_1y_1$ and $xx_2z_2y_2$.

We first show that $x$ is incident with at most two $2$-threads. Suppose that $x$ is
incident with a third $2$-thread $xx_3z_3y_3$.
Label the fourth thread incident with $x$ as $xx_4z_4y_4$, $xz_4y_4$, or $xy_4$, depending
on whether it is a $2$-thread, a $1$-thread, or a $0$-thread. Set $A =
\sst{x, x_i, z_i}{1 \leq i \leq 4}$,  $A = \sst{x,x_j, z_i}{1 \leq i
\leq 4, 1 \leq j \leq 3}$, or $A = \sst{x, x_i, z_i}{1 \leq i \leq 3}$, depending on whether $x$ is incident with four $2$-threads, a
$1$-thread, or a $0$-thread, respectively.  Since $g(G)\ge 10$, the threads do not share endvertices other than $x$, so $\delta(G-A)\ge 2$.  By the minimality of $G$, the graph $G - A$ has an
equitable $m$-coloring $f$. 

Now if $x$ is not incident with a $0$-thread,
then by Lemma~\ref{le:4-5-thread}, $f$ can be extended to $G -
\{x_1,x_2,z_1, z_2\}$ such that $f(x) \not \in \{f(y_1),f(y_2)\}$. By
 Lemma~\ref{le:2-thread}, it can be further extended to $G -
\{x_1,z_1\}$ since $f(x) \neq f(y_2)$  and to $G$ since $f(x) \neq f(y_1)$. This contradicts the choice of $G$. 

If, on the other hand, $x$ is incident with
a $0$-thread, then first extend the coloring $f$  of $G - A$ to $G -
\{x_1, z_1\} - xy_4$ such that $f(x) \not \in \{f(y_1), f(y_4)\}$ by Lemma~\ref{le:4-5-thread}.
Since $f(x) \not = f(y_4)$, it is also an equitable $m$-coloring of
$G - \{x_1,z_1\}$. Since $f(x) \not = f(y_1)$, by
Lemma~\ref{le:2-thread}, the coloring of $G - \{x_1,z_1\}$ can be
extended to $G$, a contradiction. This proves that $x$ is incident
with at most two 2-threads. 

Therefore, $t(x)\le6$, and hence $t(x)=6$. Thus, $T(x)=(2,2,0)$.
Label the two $1$-threads incident with $x$ as
$xx_3y_3$ and $xx_4y_4$.  Then $G - \sst{x, z_1,z_2, x_i}{1 \leq i \leq
4}$ has an equitable $m$-coloring. Since $y_3x_3xx_2z_2y_2$ is a
4-thread in $G - \{x_1, z_1, x_4\}$, Lemma~\ref{le:4-5-thread} implies that
$f$ can be extended to $G -  \{x_1, z_1, x_4\}$ such that $f(x) \notin \{ f(y_1), f(y_4)\}$.
By Lemma~\ref{le:2-1-thread}, it can be
further extended to $G$. This contradicts
the choice of $G$, thereby proving Claim~\ref{claim-3.3}.
\end{proof}

\begin{claim}\label{claim-3.4}
For a 3-vertex $x$, either $t(x) \leq 2$, or $T(x) =(1,2,0)$ and $m = 4$.
\end{claim}

\begin{proof}
We first prove that $T(x) \neq (1,2,0)$
if $m \geq 5$. Suppose on the contrary that $T(x) = (1,2,0)$ and $m\geq 5$. Label the two
$1$-threads incident with $x$ as $xx_1y_1$ and $xx_2y_2$ and label the
$2$-thread as $xx_3x_4y_3$. Let $A = \{x,x_1,x_2,x_3,x_4\}$.  Then
$\delta(G-A) \geq 2$ and it has an equitable $m$-coloring $f$ with color classes
$V_1,V_2, \ldots, V_m$ such that $\abs{V_1} \leq \abs{V_2}\leq \ldots \leq \abs{V_m}$.
Let $\{a,b,c,d,e\} =  \{1,2,3,4,5\}$ such that $a \neq f(y_1)$, $b \neq f(y_2)$, and $c \neq f(y_3)$.
Assigning $a$ to $x_1$, $b$ to $x_2$, $c$ to $x_4$, $d$
to $x$, and $e$ to $x_3$ yields an equitable $m$-coloring of $G$, a contradiction.

Now, we prove that if $T(x)\neq(1,2,0)$, then $t(x)\le2$. Suppose on the
contrary that $t(x)\ge3$ and $T(x)\neq(1,2,0)$.
Claim~\ref{claim-3.2} implies that $x$ is not
incident with any $t$-thread where $t \geq 3$. We first consider the
case where $x$ is not incident with a 2-thread. 
Then $T(x) = (0,3,0)$. Label the three $1$-threads incident with $x$
as $xx_iy_i$ where $d(x_i)=2$ and $d(y_i)\ge3$ for $i\in\{1,2,3\}$.
Consider the graph $G_1 = G - \{x, x_1,x_2,x_3\}$. Then
$\delta(G_1) \geq 2$, so by the minimality of $G$, the graph $G_1$ has an
equitable $m$-coloring with color classes $V_1,V_2, \ldots, V_m$ such that
$\abs{V_1} \leq \abs{V_2} \leq \ldots \leq \abs{V_m}$. Let $\{1,2,3,4\} = \{a,b,c, d\}$ such that no color in $\{a,b,c\}$ is
used by all three vertices $y_1,y_2,y_3$. An equitable $m$-coloring
of $G$ can be obtained by coloring   the vertices $x_1,x_2,x_3$ with
the colors $a,b,c$ such that no conflict occurs and coloring the
vertex $x$ with the color $d$. 
This contradiction shows that $T(x)\neq(0,3,0)$, and hence $x$ is incident
with at least one $2$-thread.

Now we consider the case that $a_2 \neq 0$. Let  $xx_1x_2y$ be
a $2$-thread incident with $x$. If $t(x) \geq 5$, then
$G-\{x_1,x_2\}$ has minimum degree 2 and has a $t$-thread $P$ that
contains $x$ for some $t\in\{4,5\}$.
Let $G_2$ be the subgraph obtained from $G - \{x_1,x_2\}$ by further deleting the degree-two vertices in $P$.
Then $G_2$ has an equitable $m$-coloring $f$. By
Lemma~\ref{le:4-5-thread}, $f$ can be extended to $G - \{x_1,x_2\}$
such that $f(x) \not = f(y)$. By Claim \thesection.3, $f$ can be
further extended to $G$. This contradiction shows that $3 \leq t(x) \leq 4$.
Since $x$ is incident with at least one $2$-thread
and $T(x) \neq (1,2,0)$, the vertex $x$ must be incident with a $0$-thread. Call it $xu$.
Since $t(x) \geq 3$, the graph $G - xu$ has a $t$-thread $P$ that
contains $x$ with $t\in\{4,5\}$. Let $G_3$ be the subgraph
obtained from $G - xu$ by further deleting the degree-two vertices
in $P$. Then $G_3$ has an equitable $m$-coloring $f$. Lemma~\ref{le:4-5-thread}
implies that $f$ can be extended to $G - xu$ such that
$f(x) \neq f(u)$. This extension of $f$ is also an
equitable $m$-coloring of $G$, a contradiction. 
This completes the proof of Claim~\ref{claim-3.4}.
\end{proof}

\medskip

A 3-vertex $x$ in $G$ is \emph{bad} if $T(x) = (1,2,0)$. Note that  if
$m\geq 5$, the configuration $T(x) = (1,2,0)$ with $d(x) = 3$ is
still reducible; thus there are no bad $3$-vertices when $m\ge 5$.
The following claim deals with reducible configurations for $m = 4$.

\begin{claim}\label{claim-3.5}
Assume $m = 4$. Let $x$ be a bad $3$-vertex and $y$ be
a vertex that is loosely $1$-adjacent to $x$.  Then\\
(1) if $d(y)=3$, then $t(y) = 1$;\\
(2) if $d(y) = 4$, then $y$ is loosely 1-adjacent to exactly one bad $3$-vertex, namely $x$.
 \end{claim}

\begin{proof} Label the threads incident with $x$ as $xx_1x_2u_1$,
$xx_3u_2$, and $xx_4y$. Here and after in the proof, we always assume that the vertices $x_i$, $y_i$ and $z_i$ have degree $2$, while the vertices $u_i$ have degree at least $3$.

 (1) Suppose that $d(y) = 3$ and $t(y) \geq 2$. Then, Claim~\ref{claim-3.4} ensures that either $t(y) = 2$, or $y$ is a bad $3$-vertex.
 If $t(y) = 2$, then $T(y) = (0,2,0)$, while if $y$ is a bad $3$-vertex,
 then $T(y) = (1,2,0)$. In either case, $y$ is incident with exactly two
 $1$-threads. Label the other 1-thread incident with $y$ as $yy_1z$.  Label the third thread
 incident with $y$ as $yu_0$ or
 $yy_2y_3u_0$ depending on  whether it is a $0$-thread or a $2$-thread.
 Set $A = \sst{x, y, y_1, x_i}{1\leq i \leq 4}$. Let $B = \emptyset$
  if $y$ is incident with a $0$-thread, and  $B = \{y_2,y_3\}$ otherwise.  Consider the graph $G_1 = G - (A\cup B)$.
  Since $g(G)\ge 10$, the vertices $u_1, u_2, z, u_0$ are distinct, thus $\delta(G_1) \geq 2$.
  By the minimality of $G$, the graph $G_1$ has an equitable $4$-coloring. Note that any $4$-equitable
coloring of $G_1$ can be extended to $G - A$, and let $f$ be an
equitable 4-coloring of $G-A$. We color $x$, $y$, $y_1$, and $x_4$ in this order as follows:
pick one color $c_1$ for $x$ in $\{1,2,3,4\}\setminus\{f(u_1), f(u_2)\}$, and $c_2$ for $y$ in $\{1,2,3,4\}\setminus\{c_1, c(y_2)\}$ if $B\neq\emptyset$ and in $\{1,2,3,4\}\setminus\{c_1, c(u)\}$ if $B=\emptyset$,  and $c_3$ for $y_1$ in $\{1,2,3,4\}\setminus\{c_1,c_2, c(z)\}$, and $c_4$ for $x_4$ in $\{1,2,3,4\}\setminus\{c_1, c_2, c_3\}$.
In such a way, $f$ can be extended to an equitable coloring of $G-\{x_1,x_2,x_3\}$ such that $f(x)\notin\{f(u_1),f(u_2)\}$.

By Lemma~\ref{le:2-1-thread}, the equitable $4$-coloring of $G - \{x_1,x_2,x_3\}$ can be further extended to $G$, which contradicts the choice of $G$, and hence prove (1).

(2)  Suppose that $d(y) = 4$ and $y$ is loosely 1-adjacent to two bad $3$-vertices $x$ and $z$.  Label the threads incident with $z$ as
$zz_1z_2u_3$, $zz_3u_4$, and $zy_1y$.  Let $u_5$ and $u_6$ be the endvertices of the two threads incident with $y$ other than the ones incident to $x$ and $z$. Set $A = \sst{x,y,z, y_1, x_i,z_j}{1 \leq i \leq 4, 1 \leq j \leq 3}$. Let $B$ be the set of 2-vertices on the two threads incident with $y$ other than $yy_1z$ and $yx_4x$.  Let $G_1 = G - (A \cup B)$. Observe that, since $g(G)\ge10$, all the named vertices are distinct except, possibly, $u_1$ and $u_3$.
Consequently, either $\delta(G_1)\ge2$, or $u_1=u_3$ and $u_1$ has degree $3$ in $G$. However, this last case contradicts Claim~\ref{claim-3.4}. Thus, $\delta(G_1) \geq 2$, and hence $G_1$ has an equitable $4$-coloring $f$ with color classes$V_1,V_2,V_3,V_4$ such that $\abs{V_1}\leq \abs{V_2} \leq \abs{V_3} \leq \abs{V_4}$.
Note that if $y$ is incident with a 2-thread, then $f$ can be extended to the $2$-vertices in the $2$-thread. This is why, in the following, we may assume without loss of generality that $y$ is not incident with a $2$-thread. 

We first consider the case where $y$ is incident with exactly two
$1$-threads: $xx_4y$ and $zy_1y$. Then $B = \emptyset$. Using
Lemma~\ref{le:4-5-thread}, we extend $f$ to $y_1yx_4x$ such that $f(y)
\notin \{f(u_5), f(u_6)\}$. Note that the colors $f(x)$, $f(x_4)$,
$f(y)$, and $f(y_1)$ are distinct. If $f(x) \in \{f(u_1), f(u_2)\}$, then
one of $f(x_4)$ and $f(y_1)$ is not in $\{f(u_1), f(u_2)\}$. If
$f(x_4) \notin \{f(u_1), f(u_2)\}$, then swap the colors of $f(x)$
and $f(x_4)$. If $f(y_1) \notin \{f(u_1), f(u_2)\}$, then swap the
colors fn $f(x)$ and $f(y_1)$. Hence we have an extension of $f$ on
$xx_4yy_1$ such that $f(x) \notin \{f(u_1), f(u_2)\}$. By
Lemmas~\ref{le:4-5-thread} and ~\ref{le:2-1-thread}, $f$ can be
further extended to $G$, a contradiction.

Now we consider the case where $y$ is incident with at least three
$1$-threads. Label the third $1$-thread as $yy_2u_5$ and the fourth
thread incident with $y$ as $yy_3u_6$ or $yu_6$ depending on whether
it is a $1$-thread or a $0$-thread. Note that either $B =\{y_2\}$ or $B = \{y_2,y_3\}$.
We first extend $f$ to $\{x_4,y,y_1\}\cup B$.
Let $a$ and $b$ be two distinct colors in
$\{1,2,3,4\}\setminus\{f(u_5),f(u_6)\}$. Assign $a$ to $y_2$.
If $B=\{y_2\}$, then assign $b$ to $y_1$, else assign $b$ to $y_3$.
Now, assign each of the colors of $\{1,2,3,4\}\setminus\{a,b\}$ arbitrarily,
making sure that both $x$ and $y_1$ are colored $1$ if $B=\{y_2,y_3\}$ and
$1\notin\{a,b\}$. This yields an equitable $4$-coloring of $G$, a contradiction.
\end{proof}

\medskip \noindent
Since $g(G)\ge 10$, we have $\mad(G)<2.5$.
Let  $M(x) = d(x) - 2.5$ be the \emph{initial charge}  of $x$
for $x \in V$. We will redistribute the charges among vertices according to the \emph{discharging rules} below.

(R1) Each $2$-vertex receives $\frac{1}{4}$ from each of the
endvertices of the thread containing it.

(R2) Each bad $3$-vertex receives $\frac{1}{4}$ from each of the
vertices that are loosely 1-adjacent to it.

Let $M'(x)$ be the charge of $x$ after application of rules R1 and R2.
The following claim shows a contradiction to (\ref{eq-0}), which implies the truth of Theorem~\ref{thm:4-coloring}.

\begin{claim}
$M'(x) \geq 0$ for each $x \in V$.
\end{claim}

\begin{proof}
Let $x\in V$.
If $d(x) = 2$, then $M'(x) = 2-2.5+\frac{2}{4} = 0$.

Assume $d(x) = 3$ and $x$ is not a bad vertex. If $x$ is not loosely $1$-adjacent to any bad vertex, then Claim~\ref{claim-3.4}
ensures that $t(x) \leq 2$, so $x$ sends out at most $2\times \frac{1}{4} = \frac{1}{2}$. If $x$ is loosely 1-adjacent to a bad vertex, then Claim~\ref{claim-3.5} implies that $t(x)=1$, so $x$ sends out at most $\frac{1}{4}+\frac{1}{4}=\frac{1}{2}$. In either case $M'(x)\geq 3-2.5-\frac{1}{2}=0$.
  
Assume $d(x) = 3$ and $x$ is a bad
  vertex.  Then $t(x) = 4$ and $x$ sends out $4\times \frac{1}{4}=1$. It also receives $\frac{1}{4}$
  from each loosely 1-adjacent vertex, of which there are $2$.   Hence  $M'(x) \geq  3 - 2.5 - 1 + 2\times \frac{1}{4}  = 0$.

Assume $d(x) = 4$.   Then $x$ is  loosely 1-adjacent to at most one
 bad 3-vertex by Claim~\ref{claim-3.5}. Hence $x$ sends out at most
 $ \frac{t(x)+1}{4} \leq \frac{3}{2}$
 since $t(x) \leq 5$ by Claim~\ref{claim-3.3}. Therefore $M'(x) \geq  4 - 2.5 - \frac{3}{2} = 0$.

Assume $d(x) \geq 5$. Let $y$ be a \kplus{3}-vertex that is loosely
$k$-adjacent to $x$.  If $k = 2$, then $x$ sends out $2\times
\frac{1}{4}$ via this 2-thread. If $k = 1$, then $x$ sends out
$\frac{1}{4}$ via this thread if $y$ is not a bad vertex and sends
out $2\times \frac{1}{4} = \frac{1}{2}$ via this 1-thread if $y$ is
a bad vertex. In summary, $x$ sends out at most $\frac{1}{2}$ via
each thread incident with it. Hence $M'(x) \geq d(x) - 2.5-
\frac{d(x)}{2} = \frac{d(x)}{2} - 2.5 \geq 0$.
\end{proof}

\section{Reduction lemmas for equitable 3-coloring}

We now proceed to equitable $3$-coloring. We first prove two lemmas which give conditions for the existence of reducible configurations.

A subdivided star $H$ is a graph obtained from a star by replacing the edges by paths (we will call these paths ``threads'' as well).  In our reducible configurations, we will see the natural connections:  if we take a vertex $v$ with the $2$-vertices on its incident threads in graph $G$, we obtain a subdivided star with root $v$.   So in the following two lemmas, even though we state and prove them as graphs, they are indeed part of the graphs under consideration.    

Let $a_i^v$ be the number of $i$-threads incident to vertex $v$. If it is clear from the context, we drop $v$ in the notation.  The two lemmas that follow give simple ways to identify reducible configurations using relations involving $a_i^v$. 

\begin{remark}
In the following lemma, the fact that we only assume two allowed colors at the root instead of three corresponds to the fact that we are allowing for one \kplus{3}-vertex adjacent to the root (i.e., one 0-thread incident with the root).
\end{remark}

\begin{lemma}[Reducing a vertex with at most one 0-thread]\label{lem:reduce-a-vertex}
Let $S$ be a subdivided star of order $s$ with root $x$.  Let $L=\{\ell_v\}$ be a list assignment to the vertices of $S$ such that $\ell_v=\{1,2,3\}$ if $v$ is neither a leaf nor the root, $\ell_v\subset \{1,2,3\}$ with $\abs{\ell_v}=2$ if $v$ is a leaf, and $\abs{\ell_x}\ge 2$.  Let $d(x)\le 6$ and assume $a_i=0$ unless $i\in\{0, 1,2,4\}$.  If $2a_4 + a_2\geq a_1+1+\varepsilon$ and $a_4\ge d(x)-4$, then $S$ is descending-equitably $L$-colorable, where $\varepsilon=3\lceil s/3\rceil - s$.
\end{lemma}

\begin{proof}
Let $c$ and $c'$ be two colors allowed at $x$. Let $p_i$ ($i\in\{1,2,3\}$) be the desired size of $V_i$.
Let $S_i$ ($i\in\{c,c'\}$) be a maximum independent set that contains the root,
and such that $i\in l_v$ for all $v\in S_i$.  Then no vertex of a $1$-thread is in
any of the $S_i$; each $2$-thread contains a leaf that is
in at least one of $S_i$'s; and for each $4$-thread, the leaf is in at least one of the
$S_i$'s, and the vertex at distance $2$ from the root is in both of the
$S_i$'s. Thus  $\abs{S_c}+\abs{S_{c'}}\ge 2+3a_4+a_2$. 

We wish first to find a color for the roots that may be extended to an independent set of size
$\lceil s/3\rceil$; the candidates for such a set are $S_c$ and $S_{c'}$.
Assume for a contradiction that they are both of size at most $\lceil s/3 \rceil - 1$.
Then, $2+3a_4+a_2 \le \abs{S_c}+\abs{S_{c'}}\le 2(\lceil s/3\rceil -1)=\frac{2}{3}(s+\varepsilon-3)=\frac{2}{3}(4a_4+2a_2+a_1+1+\varepsilon-3)$.
Therefore, $a_4+12\le a_2+2(a_1+\varepsilon+1)\le a_2+2(2a_4+a_2)$, that is $a_4+a_2\ge 4$.
So $a_2+2(a_1+\varepsilon+1)\ge a_4+12\ge 16-a_2$, and hence $a_1+a_2+\varepsilon\ge 7$. 
Adding $a_4$ to both sides and observing that $\varepsilon\leq 2$ yields
that $d(x)\ge a_1+a_2+a_4\geq 5+a_4$,
whence $a_4\leq d(x)-5$, which contradicts our hypothesis that $a_4\geq d(x)-4$.

Thus there exists an independent set of size at least $\lceil s/3 \rceil$ containing the root and having a common color available.
Let $c$ be that color, and fix a subset $T_c$ of $S_c$ of size exactly $p_c$, with the additional
property that $T_c$ contains all vertices of $4$-threads that are at distance exactly 2 from the root (this is possible because $a_4< s/3$). 

Let $c'$ and $c''$ be the other two colors. Without loss of generality
we may assume that $p_{c'}\geq p_{c''}$. We color with $c'$ first. By construction,
$T_c$ contains no vertices at distance $1$ or $3$ from the root.
There are $2a_4+a_2$ such vertices that are not leaves.
Assuming that $p_{c'}\geq \lceil s/3 \rceil$, we compare these quantities and find that $\lceil s/3\rceil = \frac{4a_4+2a_2+a_1+1+\varepsilon}{3}\leq \frac{4a_4+2a_2+2a_4+a_2}{3}\leq 2a_4+a_2$; thus at worst $2a_4+a_2$ is exactly ``big enough'' and we assign a preliminary set $W_{c'}$ the color $c'$, such that (a) all vertices at distance 3 from the root are in $W_{c'}$, and (b) some non-leaf vertices adjacent to the root are in $W_{c'}$ such that $W_{c'}$ has size $p_{c'}$ (this is possible because again $a_4 < s/3$).

The remaining vertices, which we shall group together in a set $W_{c''}$ are ``assigned'' the color $c''$, with the caveat that these $W_{c''}$ vertices contain the leaves, and thus may not have $c''$ in their list.

To pass, therefore, to a legitimate $L$-coloring, we pair the vertices
of $W_{c''}$ that are leaves with a subset of $W_{c'}$ as follows. For
each $z$ that is a leaf of a $2$-thread or a $4$-thread, define $z^*$ to be the neighbor of $z$.
For each $z$ that is a leaf of a $1$-thread, we may assign a unique
$z^*$ such that $z^*$ is a neighbor of the root and $z^*$ lies in a
4-thread (note that this is possible because from the second paragraph
of this proof, $a_2+a_4\geq 4$, whence $a_4\geq d(x)-4\ge a_4+a_2+a_1-4\geq 4+a_1-4=a_1$).
Now in every case $z\in W_{c''}$ and $z^*\in W_{c'}$ by construction;
swap the colors on $z$ and $z^*$ if $c''$ is not in $\ell_z$. Note
that the obtained coloring is legitimate,
because in each case the other vertices adjacent to $z^*$ received the color $c$.
\end{proof}

\begin{lemma}[Reducing two vertices connected by a 1-thread; one vertex may have one 0-thread]
\label{lem:reduce1thread} Suppose $x$ and $y$ are connected by a $1$-thread, and $d(x)+d(y)\le 8$. Let $S$ be a graph of order $s$ induced by the union of the subdivided star with root $x$ and the subdivided star with root $y$. Let $L=\{\ell_v\}$ be  a list assignment to the vertices of $S$ such that $\ell_v=\{1,2,3\}$ if $v$ is neither a leaf nor $y$, $\ell_v\subset \{1,2,3\}$ with $\abs{\ell_v}=2$ if $v$ is a leaf, and $\ell_y\subseteq \{1,2,3\}$ with $\abs{\ell_y}\geq 2$. Let $b_i=a_i^{x}+a_i^{y}$ for $i\in\{2,4\}$, and let $b_1=a_1^{x}+a_1^y-1$.
Then $S$ is descending-equitably $L$-colorable if $2b_4 + b_2\geq b_1-1+\varepsilon$ and $b_4\ge 1$, where $\varepsilon=3\lceil s/3\rceil - s$.
\end{lemma}
   
\begin{proof}
 Let $c$ and $c'$ be two colors allowed at $y$. Let $p_i$ ($i\in\{1,2,3\}$) be the desired size of $V_i$. Let $S_i$ ($i\in\{c,c'\}$) be a maximum independent set that contains $x$ and $y$, and such that $i\in \ell_v$ for all $v\in S_i$.
 Then no vertex of a $1$-thread is in any of the $S_i$; each $2$-thread contains a leaf that is
in at least one of the $S_i$'s; and for each $4$-thread incident with $x$ ($y$, respectively),
the  leaf is in at least one of the $S_i$'s, and the vertex at distance
$2$ from $x$ ($y$) is in both of the $S_i$'s. Thus
$\abs{S_c}+\abs{S_{c'}}\ge 4+3b_4+b_2$. 

We wish first to find a color for the root that may be extended to an independent set of size $\lceil s/3\rceil$; the candidates for such a set are $S_c$ and $S_{c'}$.
Assume for a contradiction that they are both of size at most $\lceil s/3 \rceil - 1$.
Then, $4+3b_4+b_2\le \abs{S_c}+\abs{S_{c'}}\le 2(\lceil s/3\rceil -1)=\frac{2}{3}(s+\epsilon-3)=\frac{2}{3}(4b_4+2b_2+b_1+2)$.
Therefore, $b_4+14\le b_2+2(b_1+\epsilon)\le b_2+2(2b_4+b_2)+2$, that is $b_4+b_2\ge 4$.
So $b_2+2(b_1+\epsilon)\ge b_4+14\ge 18-b_2$, and hence $b_1+b_2+\epsilon\ge 9$.
Adding $b_4$ to both sides and observing that $\varepsilon\leq 2$ yields
that $d(x)+d(y)+1=b_1+b_2+b_4+2\geq b_1+b_2+b_4+\varepsilon\geq
9+b_4\geq 10$, whence $d(x)+d(y)\geq 9$, which contradicts our assumption that $d(x)+d(y)\leq 8$.

Thus there exists an independent set of size at least $\lceil s/3 \rceil$ containing $x$ and $y$ and having a common color available.
Let $c$ be that color, and fix a subset $T_c$ of $S_c$ of size exactly
$p_c$, with the additional property that $T_c$ contains all vertices of
$4$-threads incident with $x$ or $y$ that are at distance exactly $2$ from $x$ or $y$, respectively (this is possible because $b_4< s/3$). 

Let $c'$ and $c''$ be the other two colors. Without loss of generality
we may assume that $p_{c'}\geq p_{c''}$. We color with $c'$ first. By
construction, $T_c$ contains no vertices at distance $1$ or $3$ from either $x$ or $y$.
There are $2b_4+b_2+1$ such vertices that are not leaves. Assuming that $p_{c'}\geq \lceil s/3 \rceil$,
we compare these quantities and find that $\lceil s/3\rceil =
\frac{4b_4+2b_2+b_1+2+\varepsilon}{3}\leq
\frac{4b_4+2b_2+2+2b_4+b_2+1}{3}\leq 2b_4+b_2+1$; thus at worst
$2b_4+b_2+1$ is exactly ``big enough'' and we assign a preliminary set
$W_{c'}$ the color $c'$, such that (a) all vertices that are on a
$4$-thread incident with $x$ or $y$ at a distance of $3$ from $x$ or $y$, respectively, are in $W_{c'}$; and (b) some non-leaf vertices adjacent to $x$ or $y$ are in $W_{c'}$ such that $W_{c'}$ has size $p_{c'}$ (this is possible because again $b_4 < s/3$).

The remaining vertices, which we shall group together in a set $W_{c''}$ are ``assigned'' the color $c''$, with the caveat that these $W_{c''}$ vertices contain the leaves, and thus may not have $c''$ in their list.

To pass, therefore, to a legitimate $L$-coloring, we pair the vertices
of $W_{c''}$ that are leaves with a subset of $W_{c'}$ as follows. For
each $z$ that is a leaf of a $2$-thread or a $4$-thread, define $z^*$ to
be the neighbor of $z$. For each $z$ that is a leaf of a $1$-thread, we
may assign a unique $z^*$ such that $z^*$ is a neighbor of $x$ or $y$
and $z^*$ lies in a $4$-thread, or such that $z^*$ is the neighbor of a
leaf colored by $c$ (note that this is possible because
$b_4+|\{\mbox{leafs colored $c$}\}|=p_c\geq p_{c''}$).
Now in every case $z\in W_{c''}$ and $z^*\in W_{c'}$ by construction;
swap the colors on $z$ and $z^*$ if $c''$ is not in $\ell_z$. Note that
the obtained coloring is legitimate, because in each case the other vertices adjacent to $z^*$ received the color $c$.
\end{proof}

\section{Equitable $3$-coloring}

In this section, we prove Theorem~\ref{thm:3-coloring}.

By Theorem~\ref{thm:4-coloring}, we only need to show that planar graphs
with minimum degree at least $2$ and girth at least $14$ are  equitably
3-colorable. Suppose not, and let $G$ be a counterexample with
$\abs{V} + \abs{E} $ as small as possible. The proof of the following
claim is essentially a line by line copy
of the proof of Claim~\ref{claim-3.2}, so we omit it.

\begin{claim}\label{claim-5.2}
$G$ has no $t$-thread  where $t = 3$ or $t \geq 5$, and no thread with the same endvertices.
\end{claim}

Similarly to Section 3, for a vertex $x$, let $T(x)=(a_4, a_2, a_1, a_0)$, where $a_i$ is the number of $i$-threads incident to $x$, and let $t(x)=4a_4+2a_2+a_1$.

\begin{claim}\label{claim-5.4}
Let $x$ be a vertex with $3\le d(x)\le 6$. Then \\
(a) if $d(x)=3$, then either $t(x) \leq 4$ or $T(x) =(1, 0, 2, 0)$;\\
(b) if $d(x)=4$, then $t(x)\le 7$ or $T(x)=(2, 0, 0, 2)$;\\
(c) if $d(x)\in \{5, 6\}$, then $a_4\le d(x)-2$.
\end{claim}

\begin{proof}
Assume for a contradiction that either (i) $d(x)=3$ and $t(x)\ge 5$, or (ii) $d(x)=4$ and $t(x)\ge 8$, or (iii) $d(x)\in \{5, 6\}$ and $a_4\ge d(x)-1$.

Note that if $d(x)\in \{3, 5, 6\}$, $a_0\le 1$. If $d(x)=4$, then
$a_0>1$ and $t(x)\ge 8$ only if $a_4=a_0=2$, in which case
$T(x)=(1,0,2,0)$, as wanted.
So we may assume that $a_0\le 1$, thus Lemma~\ref{lem:reduce-a-vertex} applies.

Let $H$ be the subgraph of $G$ induced by $x$ and its loosely adjacent $2$-vertices.
Then $G-H$ has an equitable $3$-coloring $f$, and we may assume that $f$ cannot be extended to $H$. Thus by Lemma~\ref{lem:reduce-a-vertex},
$2a_4+a_2\le a_1+\epsilon$, 
where $\epsilon=3\left\lceil\frac{\abs{V(H)}}{3}\right\rceil-\abs{V(H)}$. Since $t(x)=4a_4+2a_2+a_1$,
\begin{equation}\label{eq-10}
t(x)=2(2a_4+a _2)+a_1\le 3a_1+2\epsilon.
\end{equation}

Let $d(x)=3$.  By (\ref{eq-10}), $a_1\ge 1$. Then $(a_4, a_2)\in
\{(1,1), (2, 0), (1, 0), (0,2)\}$. If $(a_4, a_2)=(1,1)$, then
$\epsilon=1$ and $a_1=1$, a contradiction to (\ref{eq-10});  if $(a_4, a_2)=(2, 0)$, then $a_1=1$ and $\epsilon=2$, a contradiction to (\ref{eq-10}) again; and if $(a_4,a_2)=(0,2)$, then $a_1=1$ and $\epsilon=0$, another contradiction to (\ref{eq-10}). So $(a_4, a_2)=(1, 0)$. It follows that $a_1=2$ or $a_1=a_0=1$. If $a_1=a_0=1$, then $\epsilon=0$, a contradiction to (\ref{eq-10}). Therefore $a_1=2$ and $T(x)=(1, 0, 2, 0)$.

Let $d(x)=4$. By (\ref{eq-10}), $a_1\ge 2$.  Then $(a_4, a_2)\in
\{(1,1), (2, 0)\}$. Consequently, $a_1=2$. If $(a_4, a_2)=(1,1)$, then
$\epsilon=0$, a contradiction to~(\ref{eq-10}).
If $(a_4, a_2)=(2,0)$, then $\varepsilon=1$, again a contradiction
to~(\ref{eq-10}).

If $d(x)\in \{5, 6\}$, then $a_1\le 1$, which contradicts~(\ref{eq-10}).
\end{proof}

\medskip

We call a $3$-vertex $x$ \emph{bad} if $T(x) = (1, 0, 2,0)$

\begin{claim}\label{claim-5.5}
Let $x$ be a bad $3$-vertex.  Let $y$ be a $3$-vertex that is loosely
$1$-adjacent to $x$. Then\\
(a) $y$ is not incident to a $t$-thread where $t \geq 2$; hence $t(y) \leq 3$; and\\
(b) $x$ is the only bad $3$-vertex to which $y$ is loosely $1$-adjacent. 
\end{claim}

\begin{proof}
(a)~Suppose that $y$ is incident with a $t-$thread where $t\ge 2$. Let
$H$ be the subgraph of $G$ induced by $x$, $y$, and all $2$-vertices
loosely adjacent to $x$ or $y$.  We apply Lemma~\ref{lem:reduce1thread}
to $H$, observing that $b_4=a_4^y+1$, $b_2=a_2^y$, and $b_1=a_1^y+1$. We
find that $H$ is reducible if $2b_4+b_2\ge b_1-1+\varepsilon$, or
equivalently if $2a_4^y + a_2^y \geq a_1^y +\epsilon -2$.

Now if $a_1^y=1$, then $2a_4^y + a_2^y \geq  1 \geq
\epsilon -1 =  a_1^y + \epsilon -2$. Hence, since $y$ is adjacent to a
$t$-thread with $t\ge 2$, it must be that $a_1^y=2$.  Thus we may reduce
$H$ if $2a_4^y+a_2^y\ge \varepsilon$.  This is true if
$a_4^y>0$.  Thus we may assume that $a_2^y=1$.  But in this case
$\abs{H}=11$, $\varepsilon=1$, and $a_2^y\ge \varepsilon$. Thus $H$ is
reducible.

(b)~Suppose now that $y$ is also loosely 1-adjacent another bad
$3$-vertex $z$. Let $H$ be the subgraph induced by $x$, $y$, $z$, and all
the $2$-vertices loosely adjacent to $x$, $y$, or $z$.  Let $G'$ be $G-H$.
Note that by the girth condition, $x$ and $z$ may be loosely adjacent
to the same vertex $w$ through the $4$-threads, but in that case, $w$
cannot be a $3$-vertex, since otherwise it violates
Claim~\ref{claim-5.4}(a). So $\delta(G')\ge 2$, and thus  $G'$ is
equitably 3-colorable. We need to extend this equitable $3$-coloring to
all of $G$.  

We will $3$-color $H$, and for $i\in\{1,2,3\}$ let $U_i$ be the set of
vertices of $H$ colored by $i$.  For the coloring to remain equitable,
we need $\abs{U_1}\geq \abs{U_2}\geq \abs{U_3}\geq \abs{U_1}-1$.  Call a
proper coloring of $H$ ``good'' if it satisfies $\abs{U_1}\geq
\abs{U_2}\geq \abs{U_3}\geq \abs{U_1}-1$.

The union of $x,y,z$ together with the $1$-threads at $x$ and $z$ forms
a $9$-path; let us label it as $v_1w_1xw_2yw_3zw_4v_2$.
Label the $4$-thread at $x$ as $xx_1x_2x_3x_4v_3$, and label the $4$-thread at $z$ as $zz_1z_2z_3z_4v_4$.

First suppose that $y$ is adjacent to a $0$-thread.  Then $\abs{U_i}$
should be $5$ for all $i$, and some color is disallowed at $y$ by its
adjacency in $G$ to a vertex of $G'$.
Assume without loss of generality that $3$ is an allowed color at $y$.
Let $U_1'=\{w_1, w_4, x_1, x_3, z_3\}$, $U_2'=\{w_2, w_3, x_4, z_1,
z_4\}$, and $U_3'=\{x,y,z,x_2,z_2\}$. This is a good coloring of $H$, so
it only remains to repair any conflicts at the leaves of $H$ when $H$ is
attached to $G'$.  Notice that if there is a conflict with the leaf
adjacent to $w_1$, we may simply swap the colors on $w_1$ and $w_2$.
Likewise we may pair $w_3$ with $w_4$, $x_3$ with $x_4$, and $z_3$ with
$z_4$, swapping any pair if there is a conflict at the associated leaf.
Any such swap results in another good coloring of $H$, and swapping any pair does not interfere with any other pair.  Thus we may obtain appropriate $U_i$ in this case.

If $y$ is incident to a third 1-thread with 2-vertex $y_1$, then we keep
the $U_i'$s as before and color $y_1$ by $1$. Note that $y_1$ and $z_1$
form another swappable pair if there is a conflict at $y_1$.

By (a), $y$ is not incident to any $t$-thread with $t\geq 2$, so the proof is complete.
\end{proof}

\medskip \noindent
Since  $g(G)\ge 14$, we have $\mad(G)<\frac{7}{3}$. Let $M(x) = d(x) - \frac{7}{3}$ be the {\it initial charge}  of $x$ for $x \in V$. We will re-distribute the charges among vertices according to the {\it discharging rules} below:

(R1) Every \kplus{3}-vertex sends $\frac{1}{6}$ to each loosely adjacent $2$-vertex;

(R2) Every \kplus{3}-vertex sends $\frac{1}{6}$ to each loosely $1$-adjacent bad $3$-vertex.

\medskip

Let $M'(x)$ be the final charge of $x$.  The following claim shows a contradiction to (\ref{eq-0}), which in turn implies the truth of Theorem~\ref{thm:3-coloring}.

\begin{claim}
For each $x\in V$, $M'(x)\ge 0$.
\end{claim}
\begin{proof}
If $d(x)=2$, then $M'(x)=2-\frac{7}{3}+2\cdot \frac{1}{6}=0$.

If $d(x)=3$, then if $x$ is bad, it gains $\frac{1}{6}$ from each of the
two loosely $1$-adjacent vertices, thus $M'(x)\ge 3-\frac{7}{3}-6\cdot
\frac{1}{6}+2\cdot \frac{1}{6}=0$; if $x$ is not bad and is not loosely
$1$-adjacent to a bad vertex, then $M'(x)\ge 3-\frac{7}{3}-4\cdot
\frac{1}{6}=0$ by Claim~\ref{claim-5.4}; if $x$ is not bad and is loosely $1$-adjacent to a bad
$3$-vertex, then $t(x)\le 3$ by Claim~\ref{claim-5.5}, thus $M'(x)\ge 3-\frac{7}{3}-3\cdot
\frac{1}{6}- \frac{1}{6}=0$.

For $d(x)\ge 4$, note that $M'(x) \ge d(x) -
\frac{7}{3} - \frac{(4a_4 + 2a_2 + 2a_1)}{6}$. Since $d(x)=a_4+a_2+a_1+a_0$,
\[
M'(x)\ge \frac{1}{3}(2d(x)-7-a_4+a_0)\,.
\]

If $d(x)\ge 7$, then $M'(x)\ge (d(x)-a_4+a_0)/3\ge 0$.  If $d(x)\in \{5, 6\}$,then Claim~\ref{claim-5.4}
implies that $a_4\le d(x)-2$, thus $M'(x)\ge 0$.

Assume now that $x$ has degree $4$.  To show that $M'(x)\ge 0$, it suffices to show that $a_4\le a_0+1$, which is true, since Claim~\ref{claim-5.4} ensures that $a_4\le 1$, or $(a_4,a_0)=(2,2)$.
\end{proof}

\section{Acknowledgement}
We are heartily grateful to the anonymous referees for their very helpful comments.

\end{document}